\documentclass[12pt, a4paper]{amsart}
\usepackage{hyperref, fullpage, amsmath, amsfonts, amsthm, amssymb, stmaryrd}

\usepackage[all]{xy}

\theoremstyle{plain}
\newtheorem{Theorem}{Theorem}[section]
\newtheorem{Lemma}[Theorem]{Lemma}
\newtheorem{Corollary}[Theorem]{Corollary}
\newtheorem{Proposition}[Theorem]{Proposition}

\theoremstyle{remark}

\newtheorem{Example}[Theorem]{Example}
\newtheorem{Definition}[Theorem]{Definition}
\newtheorem{Construction}[Theorem]{Construction}
\newtheorem{Setting}[Theorem]{Setting}

\address{Fachbereich Mathematik und Statistik, University of 
Konstanz, 78457 Konstanz, Germany}
\email{arno.fehm@uni-konstanz.de}
\email{alex.prestel@uni-konstanz.de}

\begin{document}

\title{Uniform definability of henselian valuation rings in the Macintyre language}
\author{Arno Fehm and Alexander Prestel}
\thanks{\today}

\begin{abstract}
We discuss definability of henselian valuation rings in the Macintyre language $\mathcal{L}_{\rm Mac}$,
the language of rings expanded by $n$-th power predicates.
In particular, we show that henselian valuation rings with finite or Hilbertian residue field
are uniformly $\exists$-$\emptyset$-definable in $\mathcal{L}_{\rm Mac}$,
and henselian valuation rings with value group $\mathbb{Z}$ are uniformly $\exists\forall$-$\emptyset$-definable in the ring language,
but not uniformly $\exists$-$\emptyset$-definable in $\mathcal{L}_{\rm Mac}$.
We apply these results to local fields $\mathbb{Q}_p$ and $\mathbb{F}_p((t))$, as well as to higher dimensional local fields.
\end{abstract}

\maketitle

\section{Introduction}

The question of definability of henselian valuation rings in their quotient fields goes back at least to Julia Robinson, who observed that
the ring of $p$-adic integers $\mathbb{Z}_p$ can be characterized inside the field of $p$-adic numbers $\mathbb{Q}_p$ 
purely algebraically, for example for odd prime numbers $p$ as
$$
 \mathbb{Z}_p = \left\{x\in\mathbb{Q}_p : (\exists y\in\mathbb{Q}_p)(y^2=1+px^2)\right\}.
$$
This definition of the henselian valuation ring of the local field $\mathbb{Q}_p$
is existential (or diophantine) and parameter-free ($\exists$-$\emptyset$, for short),
and it depends on $p$.
For the local fields $\mathbb{F}_p((t))$, an existential parameter-free definition of the henselian valuation ring $\mathbb{F}_p[[t]]$
is much less obvious and was given only recently in \cite{AnscombeKoenigsmann}. Also this definition depends heavily on $p$.

Of particular importance in this subject
and in applications to diophantine geometry and the model theory of fields
is the question whether there are {\em uniform} definitions,
for example of $\mathbb{Z}_p$ in $\mathbb{Q}_p$ independent of $p$, and how complex such definitions have to be.
It is known (see for example \cite{Cluckersetal})
that there cannot be a uniform {\em existential} definition of $\mathbb{Z}_p$ in $\mathbb{Q}_p$ in the ring language $\mathcal{L}_{\rm ring}=\{+,-,\cdot,0,1\}$,
but partial uniformity results were obtained in \cite{Fehm}.
Similarly, partially uniform existential definitions of valuation rings of $\mathbb{Q}$ play a crucial role in the celebrated work \cite{Koenigsmann}.

Although the natural language to pose such questions is the ring language,
in the study of the theory of $\mathbb{Q}_p$ also the so-called {\em Macintyre language}
$$
 \mathcal{L}_{\rm Mac}=\mathcal{L}_{\rm ring}\cup\left\{ P_n : n\in\mathbb{N}\right\},
$$
where each $P_n$ is a unary predicate symbol interpreted as the subset of $n$-th powers of the field, occurs naturally, cf.~\cite{PR}.
In this language, the following definition has recently been obtained in \cite[Theorem 3]{Cluckersetal} using results from the model theory of pseudo-finite fields:

\begin{Theorem}[Cluckers-Derakhshan-Leenknegt-Macintyre]\label{thm:CDLM}
There is an $\exists$-$\emptyset$-formula in $\mathcal{L}_{\rm Mac}$ that defines the valuation ring of every henselian valuation
with residue field finite or pseudo-finite of characteristic not $2$.
\end{Theorem}

We recall that a field is pseudo-finite if it is perfect, pseudo-algebraically closed and has absolute Galois group $\hat{\mathbb{Z}}$.
Since one can eliminate the predicates $P_n$ by introducing new quantifiers, 
every $\mathcal{L}_{\rm Mac}$-definition gives rise to an $\mathcal{L}_{\rm ring}$-definition.
In particular, we have the following special case:

\begin{Corollary}\label{cor:CDLM}
There is an $\exists\forall$-$\emptyset$-formula in $\mathcal{L}_{\rm ring}$ that defines $\mathbb{Z}_p$ in $\mathbb{Q}_p$
and $\mathbb{F}_p[[t]]$ in $\mathbb{F}_p((t))$
for all odd prime numbers $p$.
\end{Corollary}

The aim of this note is to discuss uniform definability of henselian valuation rings in the Macintyre language
for families containing the local fields $\mathbb{Q}_p$ and $\mathbb{F}_p((t))$.
Our results exploit both
their specific (finite) residue fields and their (discrete) value groups:

A first generalization of Corollary \ref{cor:CDLM} was already given by the second author in \cite[Theorem 1]{Prestel}.
Using an adaption of the machinery developed there, we prove a definability result for $p$-henselian valuations in
the Macintyre language (see Theorem \ref{thm:phens}),
which in particular implies the following generalization of Theorem \ref{thm:CDLM}:

\begin{Theorem}\label{thm:E}
There is an $\exists$-$\emptyset$-formula in the language $\mathcal{L}_{\rm Mac}$ that defines
the valuation ring of every henselian valuation with residue field of characteristic not $2$ which is finite, pseudo-algebraically closed but not $2$-closed, or Hilbertian.
\end{Theorem}

Note that ``pseudo-algebraically closed but not $2$-closed'' includes all pseudo-finite fields,
and ``Hilbertian'' includes in particular all global fields.

In another direction, we generalize Corollary \ref{cor:CDLM} by exploiting that the henselian
valuations on $\mathbb{Q}_p$ and $\mathbb{F}_p((t))$ 
have value group $\mathbb{Z}$.
Here, an old result of Ax \cite{Ax} shows that there is a uniform $\exists\forall\exists\forall$-$\emptyset$-definition
in $\mathcal{L}_{\rm ring}$ for such valuations. 
We again work with $p$-henselian valuations and prove a result (Proposition \ref{prop:Z}) that in particular
improves Ax' definition from $\exists\forall\exists\forall$ to $\exists\forall$:

\begin{Theorem}\label{thm:Z}
There is an $\exists\forall$-$\emptyset$-formula in the language $\mathcal{L}_{\rm ring}$ that defines
the valuation ring of every henselian valuation with value group $\mathbb{Z}$.
\end{Theorem}

We also show that in this generality, the result cannot be improved further to give an existential definition in the Macintyre language
(see Proposition \ref{prop:noE}):

\begin{Theorem}\label{thm:noE}
The $t$-adic henselian valuation on $\mathbb{C}((t))$ with value group $\mathbb{Z}$ cannot be defined
by an $\exists$-$\emptyset$-formula in the language $\mathcal{L}_{\rm Mac}$.
\end{Theorem}

Finally, we also prove a variant (again for $p$-henselian valuations) that includes assumptions both on the residue field and on the value group
(Theorem \ref{thm:A}).
It implies, in particular, that the $t$-adic valuation on $\mathbb{C}((t))$ can be defined by an $\forall$-$\emptyset$-formula in $\mathcal{L}_{\rm Mac}$, and it also implies the following:

\begin{Theorem}\label{thm:AZhat}
There is an $\forall$-$\emptyset$-formula in $\mathcal{L}_{\rm Mac}$ that defines
the valuation ring of every henselian valuation with value group $\mathbb{Z}$ and
residue field $F$ of characteristic not $2$ with absolute Galois group $G_F\cong\hat{\mathbb{Z}}$.
\end{Theorem}

Combining this with Theorem \ref{thm:CDLM} (or Theorem \ref{thm:E}), we summarize:

\begin{Corollary}
There are $\exists$-$\emptyset$ and $\forall$-$\emptyset$-formulas in $\mathcal{L}_{\rm Mac}$
that define $\mathbb{Z}_p$ in $\mathbb{Q}_p$ and $\mathbb{F}_p[[t]]$ in $\mathbb{F}_p((t))$ for every odd prime $p$,
although there are no such $\exists$-$\emptyset$ or $\forall$-$\emptyset$-formulas in $\mathcal{L}_{\rm ring}$.
\end{Corollary}

Since again we can eliminate the predicates $P_n$, 
we observe that Corollary \ref{cor:CDLM} holds with $\exists\forall$ replaced by $\forall\exists$,
which can be deduced also from \cite[Theorem 2]{Prestel}.

Combining our positive and negative results we acquire an almost complete understanding of the $\mathcal{L}_{\rm Mac}$-definability 
of henselian valuations on higher-dimensional local fields in the sense of Parshin and Kato.
We briefly discuss this in Section \ref{sec:higherlocalfields}.

\section{Uniform definitions in the Macintyre language}

We will make use of the following general definability principle:

\begin{Proposition}\label{prop:Prestel}
Let $\mathcal{L}$ be a language containing $\mathcal{L}_{\rm ring}$.
Let $\Sigma$ be a first order axiom system in $\mathcal{L}\cup\{\mathcal{O}\}$, where $\mathcal{O}$ is a unary predicate symbol.
Then there exists an
$\mathcal{L}$-formula $\varphi(x)$, defining uniformly
in every model $(K, \mathcal{O})$ of $\Sigma$ the set $\mathcal{O}$, 
of quantifier type
\begin{align*}
\exists & \textrm{ iff }(K_1 \leq K_2 \Rightarrow \mathcal{O}_1
\subseteq \mathcal{O}_2) \\
\forall & \textrm{ iff }(K_1 \leq K_2 \Rightarrow \mathcal{O}_2 \cap K_1
\subseteq \mathcal{O}_1) \\
\exists\forall & \textrm{ iff }(K_1 \prec_\exists K_2 \Rightarrow \mathcal{O}_1
\subseteq \mathcal{O}_2) \\
\forall\exists & \textrm{ iff }(K_1 \prec_\exists K_2 \Rightarrow \mathcal{O}_2
\cap K_1 \subseteq \mathcal{O}_1) 
\end{align*}
for all models $(K_1, \mathcal{O}_1)$, $(K_2, \mathcal{O}_2)$ of $\Sigma$.
Here, $K_1\leq K_2$ means that $K_1$ is an $\mathcal{L}$-substructure of $K_2$,
and $K_1\prec_\exists K_2$ means that $K_1$ is existentially closed in $K_2$, as $\mathcal{L}$-structures.
\end{Proposition}

\begin{proof}
The detailed proof given in \cite{Prestel} for the special case $\mathcal{L}=\mathcal{L}_{\rm ring}$ 
goes through verbatim for arbitrary $\mathcal{L}\supseteq\mathcal{L}_{\rm ring}$.
\end{proof}

In particular, for the Macintyre language this implies:

\begin{Corollary}\label{cor:Prestel}
Let $\Sigma$ be a first order theory of fields in $\mathcal{L}_{\rm ring}\cup\{\mathcal{O}\}$, where $\mathcal{O}$ is a unary predicate symbol,
and let $N\subseteq\mathbb{N}$.
Then there exists an
$\exists$-$\emptyset$-formula (resp.\ $\forall$-$\emptyset$-formula) $\varphi(x)$ in $\mathcal{L}_{\rm ring}\cup\{P_n:n\in N\}$, defining uniformly
in every model $(K, \mathcal{O})$ of $\Sigma$ the set $\mathcal{O}$, 
if and only if 
$\mathcal{O}_1\subseteq\mathcal{O}_2$ (resp.\ $\mathcal{O}_2\cap K_1\subseteq\mathcal{O}_1$)
for all models $(K_1,\mathcal{O}_1)$, $(K_2,\mathcal{O}_2)$ of $\Sigma$
for which $K_1$ is a subfield of $K_2$, and for all $n\in N$,
$(K_2^\times)^n\cap K_1 = (K_1^\times)^n$.
\end{Corollary}

Note that the condition $(K_2^\times)^n\cap K_1 = (K_1^\times)^n$ is satisfied in particular
when $K_1$ is relatively algebraically closed in $K_2$.

We fix some notation and recall a few definitions:

\begin{Definition}
Let $K$ be a field and $v$ a (Krull) valuation on $K$.
We denote by $\mathcal{O}_v$ the valuation ring of $v$,
by $\mathfrak{m}_v$ its maximal ideal, by $\bar{K}_v$ the residue field, and by $\Gamma_v=v(K^\times)$ the 
(additively written) value group of $v$.
The valuation $v$ is {\bf henselian} if it has a unique extension to an algebraic closure $K^{\rm alg}$ of $K$,
and {\bf $p$-henselian}, for $p$ a prime number, if it has a unique extension to 
the maximal Galois pro-$p$ extension $K(p)$ of $K$, cf.~\cite[\S4.1]{EP}.
We denote by $\zeta_p$ a primitive $p$-th root of unity.
\end{Definition}

\begin{Lemma}\label{lem:relpclosed}
Let $(K_1,v_1)\leq(K_2,v_2)$ be an extension of valued fields with ${\rm char}((\bar{K_1})_{v_1})\neq p$, $\zeta_p\in K_1$ and $(K_2^\times)^p\cap K_1=(K_1^\times)^p$. If $v_2$ is $p$-henselian, then so is $v_1$.
\end{Lemma}

\begin{proof}
Under the assumptions, $v_i$ is $p$-henselian if and only if $1+\mathfrak{m}_{v_i}\subseteq(K_i^\times)^p$, cf.~\cite[Corollary 4.2.4]{EP}.
So if $v_2$ is $p$-henselian, then $1+\mathfrak{m}_{v_2}\subseteq (K_2^\times)^p$,
hence $1+\mathfrak{m}_{v_1}=(1+\mathfrak{m}_{v_2})\cap K_1\subseteq(K_2^\times)^p\cap K_1=(K_1^\times)^p$,
which implies that $v_1$ is $p$-henselian.
\end{proof}

Trivially, every henselian valuation is $p$-henselian for every $p$.
The following two propositions generalize well-known results for henselian fields.
Alternative proofs recently appeared in \cite{JK14}:

\begin{Proposition}\label{prop:hilbertian}
If $v$ is a non-trivial $p$-henselian valuation on a field $F$ with ${\rm char}(\bar{F}_v)\neq p$ and $\zeta_p\in F$,
then $F$ is not Hilbertian\footnote{For the definition of a Hilbertian field, see \cite[Ch.~12]{FriedJarden}.}.
\end{Proposition}

\begin{proof}
The proof of \cite[Lemma 15.5.4]{FriedJarden} for henselian fields goes through in the $p$-henselian setting:
Choose $a\in F$ with $v(a)>0$ and let $f(T,X)=X^p+aT-1$, $g(T,X)=X^p+T^{-1}-1$.
If $F$ is Hilbertian, then, since $f$ and $g$ are irreducible, there exists $t\in F$ such that
$f(t,X)$, $g(t,X)$ have no zero in $F$. However, both polynomials split over $F(p)$, 
and at least one of them is in $\mathcal{O}_v[X]$ and has a simple zero in the residue field,
hence has a zero in $F$, \cite[Theorem 4.2.3]{EP}.
\end{proof}

\begin{Proposition}\label{prop:PAC}
If $F$ is PAC\footnote{For the definition of a pseudo-algebraically closed (PAC) field, see \cite[Ch.~11]{FriedJarden}.} and $v$ is a
non-trivial $p$-henselian valuation on $F$, then $F(p)=F$.
\end{Proposition}

\begin{proof}
Since $v$ is $p$-henselian, it has a unique extension to $F(p)$, which we again denote by $v$.
Since $F$ is PAC it is $v$-dense in $F(p)$, see \cite[11.5.3]{FriedJarden}.
Thus, for $\sigma\in{\rm Gal}(F(p)|F)$ and $x\in F(p)$,
for every $\gamma\in\Gamma_v$
there exists $a\in F$ with $v(a-x)>\gamma$,
hence $v(a-x^\sigma)=v^\sigma(a-x)>\gamma$,
since $v^\sigma=v$. 
Together, this implies that $v(x-x^\sigma)>\gamma$,
and as this holds for all $\gamma$, we conclude that $x=x^\sigma$.
Thus, $F=F(p)$.
\end{proof}

\begin{Theorem}\label{thm:phens}
For every prime number $p$ there is an $\exists$-$\emptyset$-formula in $\mathcal{L}_{\rm ring}\cup\{P_p\}$ that defines
the valuation ring of every $p$-henselian valued field $(K,v)$ with $\zeta_p\in K$ and residue field $F$ with
${\rm char}(F)\neq p$ and
\begin{enumerate}
\item[(a)] $F$ is finite, or
\item[(b)] $F$ is PAC and $F(p)\neq F$, or
\item[(c)] $F$ is Hilbertian.
\end{enumerate}
\end{Theorem}

\begin{proof}
The valued fields as in the statement of the theorem form an elementary class 
axiomatized by some theory $\Sigma$:
The class of $p$-henselian valued fields $(K,v)$ with $\zeta_p\in K$ and residue field $F$ with
${\rm char}(F)\neq p$ can be axiomatized for example using \cite[Corollary 4.2.4]{EP}.
Moreover, the class of finite or pseudo-finite fields (which is a subclass of (a) and (b)) is elementary, as are the fields in (b) and (c).

We want to apply Corollary \ref{cor:Prestel} to $\Sigma$.
To this end, let $(K_1,v_1)$ and $(K_2,v_2)$ be such fields with $K_1$ a subfield of $K_2$ and
$(K_2^\times)^p\cap K_1=(K_1^\times)^p$.
Denote by $w$ the restriction of $v_2$ to $K_1$.
By Lemma \ref{lem:relpclosed}, $w$ is $p$-henselian.

The residue field $F_1$ of $(K_1,v_1)$ satisfies $F_1(p)\neq F_1$ in each of the cases (a)-(c):
This is obvious in case (a), holds by assumption in case (b), and is well-known in case (c), see e.g.~\cite[16.3.6]{FriedJarden}.
Thus, $v_1$ and $w$ are comparable by \cite[Proposition 3.1]{Koe95}.

If $w$ is {\em strictly finer} than $v_1$, then it induces a non-trivial $p$-henselian valuation on $F_1$,
which is a contradiction in each of the cases (a)-(c):
In case (a) because finite fields admit no non-trivial valuations at all, in case (b) by Proposition \ref{prop:PAC},
and in case (c) by Proposition \ref{prop:hilbertian}.
Therefore, $w$ is {\em coarser} than $v_1$, i.e.\ $\mathcal{O}_{v_1}\subseteq\mathcal{O}_w\subseteq\mathcal{O}_{v_2}$, as was to be shown.
\end{proof}

Since every henselian valuation is $p$-henselian for every $p$,
Theorem \ref{thm:E} now follows from the special case $p=2$.

\begin{Theorem}\label{thm:A}
Let $p$ be a prime number and $n\in\mathbb{Z}_{\geq0}$.
There is an $\forall$-$\emptyset$-formula in $\mathcal{L}_{\rm ring}\cup\{P_p\}$ that defines
the valuation ring of every $p$-henselian valued $(K,v)$ with $\zeta_p\in K$, residue field $F$ that satisfies ${\rm char}(F)\neq p$ and $|F^\times/(F^\times)^p|=p^n$, and value group that does not contain a $p$-divisible convex subgroup.
\end{Theorem}

\begin{proof}
Again, these valued fields form an elementary class axiomatized by some theory $\Sigma$, as above.
We want to apply Corollary \ref{cor:Prestel} to $\Sigma$.
Let $(K_1,v_1)$, $(K_2,v_2)$ be models of $\Sigma$ with 
$K_1\subseteq K_2$ and $(K_2^\times)^p\cap K_1=(K_1^\times)^p$,
and denote by $w$ the restriction of $v_2$ to $K_1$.
By Lemma \ref{lem:relpclosed}, $w$ is $p$-henselian.
Denote by $F_1$ and $F_2$ the residue fields of $v_1$ resp.~$v_2$.
By assumption, ${\rm dim}_{\mathbb{F}_p}(F_1^\times/(F_1^\times)^p)={\rm dim}_{\mathbb{F}_p}(F_2^\times/(F_2^\times)^p)=n$.

If $v_1$ and $w$ are incomparable then they have a common coarsening with $p$-closed residue field $F_0$, see \cite[Proposition 3.1]{Koe95}.
The convex subgroup of $\Gamma_{v_1}$ corresponding to the valuation induced by $v_1$ on $F_0$ is then $p$-divisible (as $F_0^\times=(F_0^\times)^p$),
contradicting the assumption.

If $w$ is a proper coarsening of $v_1$, then the valuation $\bar{v}_1$ 
induced by $v_1$ on the residue field $F$ of $w$
has value group a convex subgroup of $\Gamma_{v_1}$, hence not $p$-divisible.
Therefore, 
$$
 \dim_{\mathbb{F}_p}(F^\times/(F^\times)^p)\geq\dim_{\mathbb{F}_p}(\Gamma_{\bar{v}_1}/p\Gamma_{\bar{v}_1})+
 \dim_{\mathbb{F}_p}(\bar{F}_{\bar{v}_1}^\times/(\bar{F}_{\bar{v}_1}^\times)^p)
>\dim_{\mathbb{F}_p}(F_1^\times/(F_1^\times)^p)=n.
$$
Since $(K_2^\times)^p\cap K_1=(K_1^\times)^p$ and $v_2$ is $p$-henselian, also $(F_2^\times)^p\cap F=(F^\times)^p$:
Indeed, if $x\in\mathcal{O}_w^\times$ with $\bar{x}=\bar{y}^p$, $y\in \mathcal{O}_{v_2}^\times$, then, since $f(T)=T^p-x$ splits in $K_2(p)$
and $\bar{f}(T)$ has the simple zero $\bar{y}$,
there is $z\in K_2^\times$ with $z^p=x$, so $x\in (K_2^\times)^p\cap K_1=(K_1^\times)^p$, and thus $\bar{x}\in(F^\times)^p$.
Therefore, $\dim_{\mathbb{F}_p}(F_2^\times/(F_2^\times)^p)\geq\dim_{\mathbb{F}_p}(F^\times/(F^\times)^p)>n$, contradicting the assumption.
Thus, $w$ is finer than $v_1$, i.e.~$\mathcal{O}_{v_2}\cap K_1=\mathcal{O}_w\subseteq\mathcal{O}_{v_1}$, as was to be shown.
\end{proof}

For the $t$-adic valuation on $K=\mathbb{C}((t))$,
Theorem \ref{thm:A} immediately applies with $n=0$ and arbitrary $p$.
Moreover, Theorem \ref{thm:AZhat} follows from the special case $n=1$ and $p=2$ of Theorem \ref{thm:A},
since $G_F\cong\hat{\mathbb{Z}}$ implies that $|F^\times/(F^\times)^2|=2$.

We note that while every $\exists$-$\emptyset$-definition of a valuation ring with finite residue field $\mathbb{F}_q$ gives rise to an $\forall$-$\emptyset$-definition of the same ring, see \cite[Proposition 3.3]{AnscombeKoenigsmann}, it does not seem that this can be done in a uniform way, independent of $q$.

\section{Value group $\mathbb{Z}$ in the ring language}

In this section, we will prove Theorem \ref{thm:Z} in a $p$-henselian setting and for regular value groups. 

\begin{Definition}
An ordered abelian group $\Gamma$ is {\bf discrete} if it has a smallest positive element,
{\bf $p$-regular} if every quotient by a nontrivial convex subgroup is $p$-divisible,
and {\bf regular} if it is $p$-regular for every prime $p$.
It is a {\bf $\mathbb{Z}$-group} if it is discrete and regular.
\end{Definition}

An ordered abelian group $\Gamma$ is a $\mathbb{Z}$-group if and only if $\Gamma\equiv\mathbb{Z}$ as ordered groups, \cite[Theorem 4.1.3]{PrestelDelzell}.
Examples of $\mathbb{Z}$-groups are $\mathbb{Z}$ and $\mathbb{Z}\oplus\mathbb{Q}$, 
where for ordered abelian groups $\Gamma_1$, $\Gamma_2$ we denote by $\Gamma_1\oplus\Gamma_2$ the inverse lexicographic product.

For the rest of this section, we work in the following setting:

\begin{Setting}
Let $(K,v)$ be a $p$-henselian valued field
and assume that one of the following cases holds:
\begin{enumerate}
\item $\zeta_p\in K$ and ${\rm char}(\bar{K}_v)\neq p$
\item ${\rm char}(K)=p$
\item $p=2$
\end{enumerate}
We also assume that the value group $\Gamma=\Gamma_v$ is discrete
and identify its smallest nontrivial convex subgroup with $\mathbb{Z}$.
Choose an element $t\in K$ with $v(t)=1\in\mathbb{Z}\subseteq\Gamma$.
\end{Setting}

In case (1) let $f(Y)=Y^p-1$, in case (2) and (3) let $f(Y)=Y^p-Y$.
For $a\in K$ define 
$$
 R_a = \left\{ x\in K : (\exists y\in K)(f(y)=ax^p) \right\}.
$$

\begin{Lemma}\label{lem:B}
$R_a$ contains all $x\in K$ with $pv(x)>-v(a)$.
\end{Lemma}

\begin{proof}
If $pv(x)>-v(a)$, then $v(ax^p)>0$, so the reduction of $f(Y)-ax^p$ has the simple zero $y=1$.
In case (1), the splitting field of $f(Y)-ax^p$ is a Kummer extension of $K$ contained in $K(p)$;
in case (2), the splitting field $f(Y)-ax^p$ is an Artin-Schreier extension of $K$ contained in $K(p)$;
in case (3), the splitting field of $f(Y)-ax^p$ is either $K$ or a quadratic extension of $K$, hence contained in $K(p)$.
Thus, in each case, the fact that $v$ is $p$-henselian implies that there exists $y\in K$ with $f(y)=ax^p$, cf.\ \cite[Theorem~4.2.3(2)]{EP}.
\end{proof}

\begin{Lemma}\label{lem:Rt}
$R_t=\mathcal{O}_v$
\end{Lemma}

\begin{proof}
By Lemma \ref{lem:B}, $\mathcal{O}_v\subseteq R_t$.
If $x,y\in K$ satisfy $f(y)=tx^p$, then $x\in\mathcal{O}_v$.
Indeed, otherwise $v(tx^p)<0$.
In case (1), $v(1+tx^p)=v(tx^p)\equiv 1\mbox{ mod }p\Gamma$, contradicting $v(y^p)\equiv 0\mbox{ mod }p\Gamma$.
In case (2) and (3), $v(tx^p)\equiv 1\mbox{ mod }p\Gamma$, but $v(y)<0$, so $v(f(y))=v(y^p-y)=v(y^p)\equiv 0\mbox{ mod }p\Gamma$,
a contradiction.
\end{proof}

For a subset $X\subseteq K$ let $[X]^n$ denote the set $\{x_1\cdots x_n\;:\; x_1,\dots,x_n\in X\}$. Define 
$$
 A=\left\{a\in K^\times \;:\; 1\in R_a\;\mbox{ and }\; a^{-1}\notin [R_a]^{p^2}\right\}.
$$
For $\gamma\in\Gamma$ we let
$B_\gamma=\{x\in K:v(x)\geq \gamma\}$.
Thus, $B_0=\mathcal{O}_v$ and $B_1=\mathfrak{m}_v$.
Note that $B_\delta\cdot B_\gamma=B_{\delta+\gamma}$ for all $\delta,\gamma\in\Gamma$.

\begin{Lemma}\label{lem:main}
Assume that $\Gamma$ is also $p$-regular.
If $a\in A$, then $R_a\subseteq\mathcal{O}_v$.
\end{Lemma}

\begin{proof}
Let $a\in A$.
Note that $1\in R_a$ implies that $R_a\subseteq[R_a]^{p^2}$.
We do a case distinction according to $\gamma=v(a)\in\Gamma$:

\framebox{$\gamma<0$} 
In this case, $pv(a^{-1})=-p\gamma>-\gamma$,
so Lemma \ref{lem:B} implies that $a^{-1}\in R_a\subseteq[R_a]^{p^2}$,
a contradiction.

\framebox{$\gamma=0,\dots,p$}
By Lemma \ref{lem:B}, $B_1\subseteq R_a$.
Suppose that $R_a\not\subseteq\mathcal{O}_v$,
i.e.\ there exists $b\in R_a$ with $v(b)\leq -1$.
Then 
$a^{-1}\in B_{-p}\subseteq B_{(p+1)v(b)+1}= b^{p+1}\cdot B_1\subseteq[R_a]^{p+2}\subseteq[R_a]^{p^2}$, a contradiction. 

\framebox{$\gamma>p$} 
Since $\Gamma$ is $p$-regular,
there exist $k\in\{1,\dots,p\}$ 
and $\alpha\in\Gamma$ such that $p\alpha=\gamma-k$.
Then $-p\alpha=-\gamma+k>-v(a)$,
so, by Lemma \ref{lem:B}, $B_{-\alpha}\subseteq R_a$.
Thus, $B_{-p^2\alpha}=[B_{-\alpha}]^{p^2}\subseteq[R_a]^{p^2}$.
Note that $p^2\alpha\geq \gamma$:
If $\gamma=p+1$, then $\alpha=1$, so it holds;
if $\gamma\geq p+2$, then 
$p\alpha=\gamma-k\geq\gamma-p$ implies that
$$
 p^2\alpha\geq p(\gamma-p)=\gamma+(p-1)\gamma-p^2\geq\gamma+(p-1)(p+2)-p^2=\gamma+p-2\geq \gamma.
$$ 
Thus, $a^{-1}\in B_{-p^2\alpha}\subseteq [R_a]^{p^2}$, a contradiction.
\end{proof}

\begin{Proposition}\label{prop:Z}
The $\exists_3\forall_{2p^2}$-$\emptyset$-formula $\varphi(x)$ in the language $\mathcal{L}_{\rm ring}$ given by
\begin{eqnarray*}
 (\exists a,y,y_0)(\forall y_1,\dots,y_{p^2},z_1,\dots,z_{p^2})&\big(&
  \neg(a=0)\,\wedge\, f(y)=ax^p \,\wedge\, f(y_0)=a \,\wedge \\
  &&\wedge\, \neg(az_1\cdots z_{p^2}=1 \,\wedge\, \bigwedge_{i=1}^{p^2} f(y_i)=az_i^p)\quad\big)
\end{eqnarray*}
defines $\mathcal{O}_v$ in $K$ for any $p$-henselian valued field $(K,v)$ 
with discrete $p$-regular value group
satisfying one of the three condition (1)-(3).
\end{Proposition}

\begin{proof}
Clearly,
$\varphi(K) = \bigcup_{a\in A} R_a$.
By Lemma \ref{lem:main}, this set is contained in $\mathcal{O}_v$.
Let $t\in K$ with $v(t)=1$. Then $R_t=\mathcal{O}_v$ (Lemma \ref{lem:Rt}), 
so we have 
$[R_t]^{p^2}=\mathcal{O}_v$, and hence $t^{-1}\notin[R_t]^{p^2}$.
Thus,
$t\in A$,
hence $\varphi(K)\supseteq R_t=\mathcal{O}_v$,
and therefore indeed $\varphi(K)=\mathcal{O}_v$.
\end{proof}

\begin{Corollary}\label{cor:2henselian}
There is an $\exists\forall$-$\emptyset$-formula in $\mathcal{L}_{\rm ring}$ that defines
the valuation ring of every $2$-henselian valuation with discrete $2$-regular value group.
\end{Corollary}

Since every henselian valuation is in particular $2$-henselian
and $\mathbb{Z}$ is discrete $2$-regular, this implies Theorem \ref{thm:Z}.

\begin{Corollary}
If $(K,v)$ is a henselian valued field with value group $\Gamma$ regular non-divisible,
then $\mathcal{O}_v$ is $\exists\forall$-$\emptyset$-definable in $\mathcal{L}_{\rm ring}$.
\end{Corollary}

\begin{proof}
The case where $\Gamma$ is discrete follows from Corollary \ref{cor:2henselian}.
In the case where $\Gamma$ is non-discrete, Hong \cite[Theorem 4]{Hong} gives a definition,
which one can check to be $\exists\forall$: Indeed, the set $\Psi_\epsilon$ defined there is $\exists$-$\{\epsilon\}$-definable,
thus so is $\Omega_\epsilon$, hence $\mathfrak{m}_v=\bigcap_{\epsilon\neq0}\Omega_\epsilon$ is $\forall\exists$-$\emptyset$-definable,
which finally implies that $\mathcal{O}_v=(K\setminus\mathfrak{m}_v)^{-1}$ is $\exists\forall$-$\emptyset$-definable.
\end{proof}

In fact, Hong does give a definition also in the case where $\Gamma$ is discrete, but since in that case he builds on the argument of Ax,
the definition he gets is at best $\exists\forall\exists$.
The assumption that $\Gamma$ is non-divisible is, of course, necessary.

\section{Value group $\mathbb{Z}$ in the Macintyre language}

In this section we prove our negative definability results, in particular Theorem \ref{thm:noE}.
Let $(K,v)$ be a henselian valued field with value group $\Gamma=\Gamma_v$ 
and residue field $F=\bar{K}_v$ of characteristic zero.
In order to prove 
that $\mathcal{O}_v$ is not $\exists$-$\emptyset$-definable in $\mathcal{L}_{\rm Mac}$,
it suffices to construct henselian valued fields $(K_1,v_1)$, $(K_2,v_2)$ that are elementarily equivalent to $(K,v)$
such that $K_1$ is algebraically closed in $K_2$ (since $K_1$ is then an $\mathcal{L}_{\rm Mac}$-substructure of $K_2$) 
and $\mathcal{O}_{v_1}\not\subseteq\mathcal{O}_{v_2}$.
We first recall some standard definitions and facts:

\begin{Definition}
For an ordered abelian group $\Gamma$ we denote by $F((x^\Gamma))$ the field of generalized power series $\sum_{\gamma\in\Gamma}a_\gamma x^\gamma$
with well-ordered support.
The natural power series valuation $v(\sum_{\gamma\in\Gamma}a_\gamma x^\gamma)=\min\{\gamma:a_\gamma\neq0\}$
has value group $\Gamma$, residue field $F$ and is henselian, cf.~\cite[Corollary 18.4.2]{Efrat}.
As usual, we write $F((x)):=F((x^\mathbb{Z}))$ for the field of formal Laurent series.
If $\Gamma_1,\Gamma_2$ are ordered abelian groups 
there is a natural isomorphism $F((x^{\Gamma_1\oplus\Gamma_2}))\cong F((x_1^{\Gamma_1}))((x_2^{\Gamma_2}))$.
\end{Definition}

\begin{Construction}
Let $\Delta$ be the divisible hull of $\Gamma$.
We consider the power series fields
$$
 K_1=F((x^\Delta))((t^\Gamma))
$$
with value group $u(K_1^\times)=\Delta\oplus\Gamma$
and
$$
 K_2=F((s^\Gamma))((y^\Delta))((z^\Delta))
$$
with value group $v_2(K_2^\times)=\Gamma\oplus\Delta\oplus\Delta$.
Moreover, 
let $F_1:=F((x^\Delta))$ and
denote
by $v_1$ the power series valuation on $K_1=F_1((t^\Gamma))$ with value group $\Gamma$ and residue field $F_1$.
Define an embedding $\phi$ of $K_1$ into the subfield
$$
 K_0=F((s^\Gamma))((y^\Delta))((z^\Gamma))
$$ 
of $K_2$ as follows:
For 
$$
 f=\sum_\gamma f_\gamma(x)t^\gamma\in K_1
$$
with $f_\gamma(x)\in F_1$ for all $\gamma$ let
$$
 \phi(f)=\sum_\gamma f_\gamma(y) s^\gamma z^\gamma\in K_0.
$$
This is indeed a homomorphism:
For example, we can view it as the composition $\phi=\alpha\circ\epsilon$ of the canonical embedding
$\epsilon:K_1\rightarrow K_0$ given by $\epsilon(x)=y$, $\epsilon(t)=z$,
with the automorphism $\alpha$ of $K_0$ that fixes
$F((s^\Gamma))((y^\Delta))$ and maps $\alpha(z^\gamma)=s^\gamma z^\gamma$.
$$
 \xymatrix{
 && K_2=F\overbrace{((s^\Gamma))((y^\Delta))((z^\Delta))}^{v_2}\ar@{-}[d] \\
 K_1=F\overbrace{((x^\Delta))\underbrace{((t^\Gamma))}_{v_1}}^u\ar@{->}[rr]^\phi\ar@{-}[d] && K_0=F((s^\Gamma))((y^\Delta))((z^\Gamma))\ar@{-}[d] \\
 F_1=F((x^\Delta))\ar@{-}[d] && F((s^\Gamma))((y^\Delta))\ar@{-}[d]&\\
 F && F((s^\Gamma))\ar@{-}[d] \\
 && F\\
 }
$$
\end{Construction}

\begin{Lemma}
If $F((x^\mathbb{Q}))\equiv F$ and $\Gamma\equiv\Gamma\oplus\mathbb{Q}$, then $(K_1,v_1)\equiv(K_2,v_2)\equiv(K,v)$.
\end{Lemma}

\begin{proof}
Note that $\mathbb{Q}\equiv\Delta\equiv\Delta\oplus\Delta$ since the theory of divisible ordered abelian groups is complete, cf.\ \cite[Theorem 4.1.1]{PrestelDelzell}. 
Thus $F((x^\mathbb{Q}))\equiv F((x^\Delta))$ by the Ax-Kochen-Ershov theorem \cite[Theorem 4.6.4]{PrestelDelzell},
and $\Gamma\oplus\mathbb{Q}\equiv\Gamma\oplus\Delta\oplus\Delta$,
since lexicographic products preserve elementary equivalence, cf.~\cite[proof of 3.3]{Giraudet}.
Therefore,
$$
 (\bar{K}_1)_{v_1}=F_1=F((x^\Delta))\equiv F((x^\mathbb{Q}))\equiv F=(\bar{K}_2)_{v_2} = K_v,
$$ 
and 
$$
 \Gamma_{v_2}=\Gamma\oplus\Delta\oplus\Delta\equiv\Gamma\oplus\mathbb{Q}\equiv\Gamma=\Gamma_{v_1}=\Gamma_v.
$$ 
Hence, since $(K_1,v_1)$, $(K_2,v_2)$ and $(K,v)$ are henselian valued with residue field of characteristic zero,
the Ax-Kochen-Ershov theorem implies that
$(K_1,v_1)\equiv(K_2,v_2)\equiv(K,v)$ as valued fields.
\end{proof}

\begin{Lemma}\label{lem:valres}
$\phi^{-1}(\mathcal{O}_{v_2})=\mathcal{O}_u$
\end{Lemma}

\begin{proof}
The definition of $\phi$ implies that $\phi(\mathcal{O}_u)\subseteq\mathcal{O}_{v_2}$ and $\phi(\mathfrak{m}_u)\subseteq\mathfrak{m}_{v_2}$:
Indeed, for $\epsilon:K_1\rightarrow K_0$ this statement is obvious, and $\alpha:K_0\rightarrow K_0$ leaves $\mathcal{O}_{v_2|_{K_0}}$ invariant.
It follows that $\phi(\mathcal{O}_u)=\mathcal{O}_{v_2}\cap\phi(K_1)$.
\end{proof}

\begin{Lemma}
$\phi(K_1)$ is algebraically closed in $K_2$.
\end{Lemma}

\begin{proof}
By Lemma \ref{lem:valres}, the embedding $\phi:K_1\rightarrow K_2$ induces an embedding 
$$
 \phi_*:\Gamma_u=\Delta\oplus\Gamma\rightarrow\Gamma\oplus\Delta\oplus\Delta=\Gamma_{v_2}
$$
of value groups given by $\phi_*(\delta,\gamma)=(\gamma,\delta,\gamma)$.
Observe that $\phi_*(\Gamma_u)$ is pure in $\Gamma_{v_2}$: 
Indeed, if $\phi_*(\delta,\gamma)=n(\gamma_1,\delta_1,\delta_2)$ with $\gamma_1\in\Gamma$, $\delta_1,\delta_2\in\Delta$,
then $\phi_*(\delta,\gamma)=n\phi_*(\gamma_1,\delta_1)\in n\phi_*(\Gamma_u)$.

Let $L$ be a finite extension of $K_1':=\phi(K_1)$ in $K_2$. 
The pureness of the value groups implies that $v_2$ is unramified in $L|K_1'$,
and both fields have the same residue field $F$.
So since the henselian valued field $(K_1',v_2)$ of residue characteristic zero is algebraically maximal (see \cite[Theorem 4.1.10]{EP}),
we conclude that $L=K_1'$.
\end{proof}

\begin{Proposition}\label{prop:noE}
If $(K,v)$ is a henselian valued field with value group $\Gamma_v\equiv\Gamma_v\oplus\mathbb{Q}$
and residue field $F$ of characteristic zero with $F\equiv F((\mathbb{Q}))$, then there is no $\exists$-$\emptyset$-formula in $\mathcal{L}_{\rm Mac}$
that defines the valuation ring of $v$.
\end{Proposition}

\begin{proof}
We apply the above construction and
identify $K_1$ with $\phi(K_1)\subseteq K_2$.
Then $\mathcal{O}_{v_2}\cap K_1=\mathcal{O}_u$,
and since $\mathcal{O}_u\subsetneqq\mathcal{O}_{v_1}$, this implies that $\mathcal{O}_{v_1}\not\subseteq\mathcal{O}_{v_2}$.
Thus, $(K_1,v_1)$ and $(K_2,v_2)$ satisfy all properties 
listed at the beginning of this section,
which concludes the proof.
\end{proof}

Since $\mathbb{Z}\equiv\mathbb{Z}\oplus\mathbb{Q}$ and $\mathbb{C}\cong\mathbb{C}((\mathbb{Q}))$,
Proposition \ref{prop:noE} immediately applies to $\mathbb{C}((t))$, thereby proving Theorem \ref{thm:noE}.
We will discuss more applications of Proposition \ref{prop:noE} in the next section.

\section{Higher dimensional local fields}
\label{sec:higherlocalfields}

In this last section we briefly discuss the henselian valuations on higher dimensional local fields,
by which we mean the following:

\begin{Definition}
A {\bf ($1$-dimensional) local field} is a completion of a number field (i.e.\ a field isomorphic to $\mathbb{R}$, $\mathbb{C}$
or a finite extension of $\mathbb{Q}_p$), or a completion of the function field of a curve over a finite field
(i.e.\ a field isomorphic to a finite extension of $\mathbb{F}_p((t))$).
An {\bf $n$-dimensional local field} is a complete valued field with value group $\mathbb{Z}$ and residue field
an $(n-1)$-dimensional local field.
\end{Definition}

Examples for $2$-dimensional local fields are $\mathbb{R}((t))$, $\mathbb{C}((t))$, $\mathbb{Q}_p((t))$ and $\mathbb{F}_p((t))((s))$.
An $n$-dimensional local field $K$ carries either $k=n$ or $k=n-1$ many different henselian valuations $v_1,\dots,v_k$,
where the value group of $v_k$ is a lexicographic product of $k$ copies of $\mathbb{Z}$.
(The fact that there are no other henselian valuations except for the obvious ones follows from F.K. Schmidt's theorem \cite[Theorem 4.4.1]{EP}.)

\begin{Lemma}\label{lem:regularquotient}
If an ordered abelian group $\Gamma$ has a proper convex subgroup $H$ such that $\Gamma/H$ is regular, then $\Gamma\equiv\Gamma\oplus\mathbb{Q}$.
\end{Lemma}

\begin{proof}
First of all, $\Gamma\equiv H\oplus\Gamma/H$, cf.~\cite[bottom of p.~282]{Giraudet},
and if $\Gamma/H\equiv\Gamma/H\oplus\mathbb{Q}$, then 
$$
 \Gamma\oplus\mathbb{Q}\equiv H\oplus\Gamma/H\oplus\mathbb{Q}\equiv H\oplus\Gamma/H\equiv\Gamma
$$
since lexicographic products preserve elementary equivalence \cite[proof of 3.3]{Giraudet}. 
Therefore we can assume without loss of generality that $\Gamma$ is regular.
Since regularity is preserved under elementary equivalence (as follows for example from \cite[Proposition 4]{Conrad}),
some elementary extension $\Gamma\prec\Gamma^*$ has a proper convex subgroup $H$ with $\Gamma^*/H$ divisible.\footnote{Alternatively, one could prove the regular case using the classical results of \cite{RobinsonZakon}.}
Thus, by the same reasoning as before, it suffices to prove the claim for $\Gamma$ divisible.
For $\Gamma$ divisible, also $\Gamma\oplus\mathbb{Q}$ is divisible, hence $\Gamma\equiv\Gamma\oplus\mathbb{Q}$
since the theory of divisible ordered abelian groups is complete \cite[Theorem 4.1.1]{PrestelDelzell}.
\end{proof}

\begin{Example}\label{Ex:ZZQ}
Since all archimedean groups are regular, Lemma \ref{lem:regularquotient} implies that
all ordered abelian groups $\Gamma$ of finite rank satisfy $\Gamma\equiv\Gamma\oplus\mathbb{Q}$.
This includes in particular the groups $\mathbb{Z}\oplus\dots\oplus\mathbb{Z}$
that occur as value groups of higher dimensional local fields.
\end{Example}

\begin{Example}\label{Ex:FFQ}
The condition $F\equiv F((\mathbb{Q}))$ is satisfied for the following fields $F$:
\begin{enumerate}
\item[(a)] $F$ is algebraically closed
\item[(b)] $F$ is real closed
\item[(c)] $F$ is $p$-adically closed
\item[(d)] $F$ admits a henselian valuation $v$ with $\Gamma_v=\mathbb{Z}$ and ${\rm char}(\bar{F}_v)=0$
\end{enumerate}
Indeed, in (a), $F((\mathbb{Q}))$ is again algebraically closed and the theory of algebraically closed fields of fixed characteristic is complete. Similarly for (b) and (c). In (d), applying the Ax-Kochen-Ershov theorem three times gives that 
$F\equiv \bar{F}_v((\mathbb{Z}))\equiv\bar{F}_v((\mathbb{Z}))((\mathbb{Q}))\equiv F((\mathbb{Q}))$, since $\mathbb{Z}\equiv\mathbb{Z}\oplus\mathbb{Q}$.
\end{Example}

It should now be clear that we get a complete understanding of the $\mathcal{L}_{\rm Mac}$-definability of the henselian valuations on all fields of the form $F((t_1))\dots((t_n))$ where $F$ is a local field of characteristic zero. Since the uniformity in Theorem \ref{thm:A} depends on $|F^\times/(F^\times)^2|$, which is $1$ for $F=\mathbb{C}$, $2$ for $F=\mathbb{R}$ and $4$ for $F=\mathbb{Q}_p$, we do not formulate a general result but rather discuss one family of examples in detail:

\begin{Example}
The $3$-dimensional local field $K=\mathbb{Q}_\ell((t))((s))$ has three non-trivial hen\-selian valuations:
The valuation $v_1$ with value group $\mathbb{Z}$ and residue field $\mathbb{Q}_\ell((t))$,
the valuation $v_2$ with value group $\mathbb{Z}\oplus\mathbb{Z}$ and residue field $\mathbb{Q}_\ell$,
and the valuation $v_3$ with value group $\mathbb{Z}\oplus\mathbb{Z}\oplus\mathbb{Z}$ and residue field $\mathbb{F}_\ell$.
The definability of these valuations is as follows:
\begin{center}
\begin{tabular}{|c|c|c|c|c|}
\hline
& $\exists$ in $\mathcal{L}_{\rm Mac}$ & $\forall$ in $\mathcal{L}_{\rm Mac}$ & $\exists\forall$ in $\mathcal{L}_{\rm ring}$ & $\forall\exists$ in $\mathcal{L}_{\rm ring}$ \\ \hline
$v_1$ & No (a) & Yes (d) & Yes (g)  & Yes (i) \\ \hline
$v_2$ & No (b)  & Yes (e)  & ? & Yes (i)\\ \hline
$v_3$ & Yes (c)  & Yes (f)  & Yes (h)  & Yes (i)\\ \hline
\end{tabular}
\end{center}
Here, {\em Yes} means ``uniform for all odd prime numbers $\ell$'', and {\em No} means ``not even for a fixed $\ell$''.
The question mark indicates that neither do we know that $v_2$ is $\exists\forall$-definable in $\mathcal{L}_{\rm ring}$ for any fixed $\ell$,
nor do we know that there is no such definition that works uniformly for all $\ell$.
\end{Example}

\begin{proof}
\begin{enumerate}
\item[(a)] The value group of $v_1$ is $\mathbb{Z}\equiv\mathbb{Z}\oplus\mathbb{Q}$, and the residue field of $v_1$
 is $F=\mathbb{Q}_\ell((t))$, which carries a henselian valuation with value group $\mathbb{Z}$ and residue field $\mathbb{Q}_\ell$
 of characteristic zero, hence $F\equiv F((\mathbb{Q}))$ by Example \ref{Ex:FFQ}(d). Therefore, Proposition \ref{prop:noE} applies.
\item[(b)] The value group of $v_2$ is $\Gamma_{v_2}=\mathbb{Z}\oplus\mathbb{Z}$, so $\Gamma_{v_2}\equiv\Gamma_{v_2}\oplus\mathbb{Q}$
 by Example \ref{Ex:ZZQ}. The residue field of $v_2$ is $F=\mathbb{Q}_\ell$, so $F\equiv F((\mathbb{Q}))$ by Example \ref{Ex:FFQ}(c).
 So, again Proposition~\ref{prop:noE} applies.
\item[(c)] Since $v_3$ has finite residue field $\mathbb{F}_\ell$, this follows from Theorem \ref{thm:E}.
\item[(f)] Since $v_3$ has residue field $\mathbb{F}_\ell$ and $|\mathbb{F}_\ell^\times/(\mathbb{F}_\ell^\times)^2|=2$,
 and $\Gamma_{v_3}$ is discrete (so in particular has no nontrivial $2$-divisible convex subgroup)
 we can apply Theorem \ref{thm:A} with $p=2$ and $n=1$.
\item[(e)] Since $v_2$ has residue field $\mathbb{Q}_\ell$ and $|\mathbb{Q}_\ell^\times/(\mathbb{Q}_\ell^\times)^2|=4$ by Hensel's lemma,
 and $\Gamma_{v_2}$ is discrete,
  we can apply Theorem \ref{thm:A} with $p=2$ and $n=2$.
\item[(d)] Since $v_1$ has residue field $\mathbb{Q}_\ell((t))$ and $|\mathbb{Q}_\ell((t))^\times/(\mathbb{Q}_\ell((t))^\times)^2|=8$ by Hensel's lemma, and $\Gamma_{v_1}$ is discrete, we can apply Theorem \ref{thm:A} with $p=2$ and $n=4$.
\item[(g)] Since $v_1$ has value group $\mathbb{Z}$, this is Theorem \ref{thm:Z}.
\item[(h)] This follows from the fact that there is an $\exists$-definition in $\mathcal{L}_{\rm Mac}$.
\item[(i)] This follows from the fact that there is an $\forall$-definition in $\mathcal{L}_{\rm Mac}$.
\end{enumerate}
\end{proof}

\section*{Acknowledgements}

The authors would like to thank Will Anscombe and Franziska Jahnke for helpful comments on a previous version,
and Immanuel Halupczok for some help with ordered abelian groups and the paper \cite{CluckersHalupczok}.

\end{document}